\UseRawInputEncoding
\documentclass[12pt]{amsart}
\usepackage{cases}
\usepackage{amscd}
\usepackage{amsmath}
\usepackage{amsthm}
\usepackage{amssymb}
\usepackage{mathrsfs}
\usepackage{amsfonts}
\usepackage{color}

\usepackage{pifont}

\usepackage{upgreek}
\usepackage{bm}

\usepackage{indentfirst, latexsym, bm, amsmath, eufrak, amsthm}
\usepackage{amsmath}
\usepackage{amsthm}
\usepackage{amssymb}
\usepackage{mathrsfs}
\usepackage{amsfonts}

\usepackage{pifont}

\usepackage{upgreek}
\usepackage{bm}
\newcommand*{\tr}{\mathrm{tr}}
\numberwithin{equation}{section}
\newtheorem{theo}{Theorem} 

\newtheorem{lem}{Lemma}
\newtheorem{mcor}{Corollary}
\newtheorem{remark}{Remark}
\newtheorem{prop}{Proposition}

\newcommand*{\D}[1]{\ensuremath{\nabla^{#1}}}

\begin{document}
\title{Properties of Berwald scalar curvature}

\author[Ming Li]{Ming Li$^1$}
\author[Lihong Zhang]{Lihong Zhang}



\address{Ming Li: Mathematical Science Research Center,
Chongqing University of Technology,
Chongqing 400054, P. R. China }

\email{mingli@cqut.edu.cn}

\address{Lihong Zhang: School of Science,
Chongqing University of Technology,
Chongqing 400054, P. R. China }

\email{1770218088@qq.com}


\thanks{$^1$~Partially supported by NSFC (Grant No. 11501067, 11571184, 11871126) }

\maketitle

\begin{abstract}
In this short paper, we prove that a Finsler manifold with vanishing Berwald scalar curvature has zero $\mathbf{E}$-curvature. As a consequence, Landsberg manifolds with vanishing Berwald scalar curvature are Berwald manifolds. This improves a previous result in \cite{Li}. For $(\alpha,\beta)$-metrics on manifold of dimension greater than 2, if the mean Landsberg curvature and the Berwald scalar curvature both vanish, the Berwald curvature also vanishes.
\end{abstract}


\section*{Introduction}

Let $(M,F)$ be an $n$-dimensional Finsler manifold. Let $SM=\{F(x,y)=1\}$ be the unit sphere bundle (or indicatrix bundle) with the natural projection $\pi:SM\to M$.
Let $\omega=F_{y^i}dx^i$ be the Hilbert form, which defines a contact structure on $SM$. The dual of $\omega$ with respect to the Sasaki-type metric on $SM$ is the Reeb vector filed $\xi$. It is verified that $\xi$ is the restriction of the spray $\mathbf{G}$ on $SM$. Let $\mathcal{D}$ be the contact distribution $\{\omega=0\}$. The Berwald curvature $\mathbf{B}$ as a part of the curvature endomorphism of the Berwald connection is divided into four parts. Along the Reeb vector filed direction $\mathbf{B}(\xi)=0$, however the transpose of $\mathbf{B}$ along $\omega$ gives twice the Landsberg curvature $\mathbf{B}^{t}(\omega)=2\mathbf{L}$. The remain part of $\mathbf{B}$ on $\mathcal{D}$ is in general not a symmetric endomorphism. $\mathbf{B}$ is totally symmetric on $\mathcal{D}$ when $\mathbf{L}=0$. $M$ is called a Berwald manifold if $\mathbf{B}=0$.

Let $\mathbf{E}=F\cdot\tr\mathbf{B}$ be the mean Berwald curvature or the $\mathbf{E}$-curvature. There are examples with vanishing $\mathbf{E}$-curvature which are not Berwald manifolds. In \cite{Li}, we proved that $\mathbf{L}=0$ and $\mathbf{E}=0$ implies $\mathbf{B}=0$.

It is natural to consider the scalar $\mathsf{e}:=\tr \mathbf{E}$, which will be called the \emph{Berwald scalar curvature}. In general $\mathsf{e}$ is a function on $SM$, and the symmetric tensor $\mathbf{E}$ is not isotropic. The main result of this paper is the following property about the Berwald scalar curvature.

\begin{theo}\label{theo 1}
 Let $(M,F)$ be an $n$-dimensional Finsler manifold. If $\mathsf{e}=\tr\mathbf{E}$ is a function on $M$, then $\mathbf{E}$ is isotropic, $$\mathbf{E}=\frac{1}{n-1}\mathsf{e}\mathbf{h}(\cdot,J\cdot),$$
 where $\mathbf{h}$ is the angular metric, $J$ the almost complex structure on $\mathcal{D}$. In this case, the Finsler manifold has weak isotropic $\mathbf{S}$-Curvature.
\end{theo}

A direct consequence is that vanishing Berwald scalar curvature implies vanishing $\mathbf{E}$-curvature. The following theorem improves the above mentioned result in \cite{Li}.
\begin{theo}
  Let $(M,F)$ be an $n$-dimensional Landsberg manifold. If the Berwald scalar curvature vanishes, then $M$ is a Berwald manifold.
\end{theo}

For Finsler manifolds with $(\alpha,\beta)$-metrics, we obtain the following result.
\begin{theo}\label{theo 3}
  Let $M$ be an $n\geq3$ dimensional manifold equipped an $(\alpha,\beta)$-metric $F=\alpha\phi(s)$, where $s=\frac{\beta}{\alpha}$, $\alpha$ is a Riemannian metric, $\beta$ an one form on $M$ and $\phi$ is a smooth function of real variables. If the mean Landsberg curvature $\mathbf{J}=0$ and the Berwald scalar curvature $\mathsf{e}=0$, then $M$ is a Berwald manifold.
\end{theo}
By the result in \cite{ZouCheng14}, we only need to prove Theorem \ref{theo 3} for the case $\phi$ not a polynomial. By Theorem \ref{theo 1}, $M$ has zero $\mathbf{E}$-curvature. For $(\alpha,\beta)$-metrics on a manifold of dimension $n\geq3$, by a result in \cite{ChenHePan}, vanishing  $\mathbf{E}$-curvature implies $\beta$ has constant length with respect to the Riemannian metric $\alpha$. In \cite{ZouCheng14}, it is proved that $(\alpha,\beta)$ metrics with $\mathbf{J}=0$ and $\|\beta\|_{\alpha}$ constant must be Berwald metrics. Therefore Theorem \ref{theo 3} is proved.

\medskip

This paper contains two parts. In Sect. 1, we introduce some basic results of the Chern connection and the Berwald connection in Finsler geometry. Then we prove the main theorem in Sect. 2.

\medskip

In this paper we adopt the index range
$$1\leq i,j,k,\ldots\leq n, \quad 1\leq \alpha,\beta,\gamma,\ldots\leq n-1.$$

\medskip

{\bf Acknowledgements.}  The first author
would like to thanks Professor Huitao Feng for his consistent support and encouragement.

\section{Preliminary results }

In this section, we would like to review some facts in Finsler geometry which will be used later.

Let $(M,F)$ be an $n$-dimensional Finsler manifold. Let $SM=\{F(x,y)=1\}$ be the unit sphere bundle (or indicatrix bundle) with the natural projection $\pi:SM\to M$.
Let $\{\mathbf{e}_1,\ldots,\mathbf{e}_n,\mathbf{e}_{n+1},\ldots,\mathbf{e}_{2n-1}\}$ be a local adapted orthonormal frame with respect to the Sasaki-type Riemannian metric $g^{T(SM)}$ on $T(SM)$, where $\mathbf{e}_n=\xi$ is the Reeb vector filed (c.f. \cite{Mo}). Let $\theta=\left\{\omega^1,\ldots,\omega^n,\omega^{n+1},\ldots,\omega^{2n-1}\right\}$ be the dual frame, then $\omega^n=F_{y^i}dx^i$ is the Hilbert form. $\omega^n$ defines a contact structure on $SM$. Let $\mathfrak{l}$ be the trivial bundle generated by $\xi$. Let $\mathcal{D}$ be the contact distribution $\{\omega^n=0\}$. We define an almost complex structure on $\mathcal{D}$ by $$J=-\omega^\alpha\otimes\mathbf{e}_{n+\alpha}+\omega^{n+\alpha}\otimes\mathbf{e}_\alpha.$$ Let $\mathcal{F}:=V(SM)$ be the integrable distribution given by the tangent spaces of the fibers of $SM$.  Let $p:T(SM)\to \mathcal{F} $ be the natural projection. Then the tangent bundle $T(SM)$ admits a splitting
\begin{equation}\label{splitting}
  T(SM)=\mathfrak{l}\oplus J\mathcal{F}\oplus \mathcal{F}=:H(SM)\oplus \mathcal{F}.
\end{equation}
The horizontal subbundle $H(SM)=\mathfrak{l}\oplus J\mathcal{F}$ is spanned by $\{\mathbf{e}_1,\ldots,\mathbf{e}_n\}$ on which the Chern connection is defined. And $\{\mathbf{e}_{n+1},\ldots,\mathbf{e}_{2n-1}\}$ gives a  local frame of $\mathcal{F}$.

Let $\nabla^{\rm Ch}:\Gamma(H(SM))\to\Omega^1(SM;H(SM))$ be the Chern connection, which can be extended to a map
\begin{align*}
\nabla^{\rm Ch}:\Omega^*(SM;H(SM))\rightarrow \Omega^{*+1}(SM;H(SM)),
\end{align*}
where $\Omega^*(SM;H(SM)):=\Gamma(\Lambda^*(T^*SM)\otimes H(SM))$ denotes the horizontal valued differential forms on $SM$ (refers to pp.3-5 of \cite{Zhang} for details).
It is well known that the symmetrization of Chern connection  $\hat{\nabla}^{\rm Ch}$ is the Cartan connection.
The difference between $\hat{\nabla}^{\rm Ch}$ and $\nabla^{\rm Ch}$ will be referred as the Cartan endomorphism,
\begin{align*}
H= \hat{\nabla}^{\rm Ch}-\nabla^{\rm Ch}\quad\in\Omega^1(SM,{\rm
End}(H(SM))).
\end{align*}
Set $H=H_{ij}\omega^j\otimes\mathbf{e}_i$. By Lemma 3 and Lemma 4 in \cite{FL}, $H_{ij}=H_{ji}=H_{ij\gamma}\omega^{n+\gamma}$ has the following form under natural coordinate systems
\begin{align*}
H_{ij\gamma}=-A_{pqk}u_i^pu_j^qu_{\gamma}^k,
\end{align*}
where $A_{ijk}=\frac{1}{4}F[F^2]_{y^iy^jy^k}$ and $u_i^j$ are the transformation matrix from adapted orthonormal frames to natural frames.

Set $\eta={\rm tr}[H]~\in\Omega^1(SM).$
It is referred as the Cartan-type form in \cite{FL}. The Cartan-type form has the following local formula
\begin{align*}
\eta=\sum_{i=1}^nH_{ii\gamma}\omega^{n+\gamma}=:H_{\gamma}\omega^{n+\gamma}.
\end{align*}
\begin{remark}
In literature, Cartan from $I$ is defined as $I=H_{\gamma}\omega^\gamma$. It is easy to check $I=J^*\eta$, where $J^*$ denote the dual of the almost complex structure $J$ mentioned above. This is the reason why we call $\eta$ the Cartan-type form. The Cartan-type form $\eta$ is crucial for our study.
\end{remark}

Let $\bm{\omega}=(\omega_j^i)$ be the connection matrix of the Chern
connection with respect to the local adapted orthonormal frame field,
i.e.,
\begin{align*}
\D{\rm Ch}\mathbf{e}_i=\omega_i^j\otimes\mathbf{e}_j.
\end{align*}
\begin{lem}[\cite{BaoChernShen,ChernShen,Mo,SS}]\label{sturcture eq}
The connection matrix $\bm{\omega}=(\omega_j^i)$ of $\nabla^{\rm Ch}$ is determined by the following structure equations,
\begin{equation}\left\{
\begin{aligned}
&d\vartheta=-\bm{\omega}\wedge\vartheta,\\
&\bm{\omega}+\bm{\omega}^t=-2H,
\end{aligned}\right.\label{Chern connection structure eq. matrix}
\end{equation}
where $\vartheta=(\omega^1,\ldots,\omega^{n})^t$. Furthermore,
$$\omega_{\alpha}^{n}=-\omega^{\alpha}_{n}=\omega^{n+\alpha},\quad{\rm and}\quad \omega^{n}_{n}=0.$$
\end{lem}

\begin{remark}
In \cite{FL}, we proved that the Chern connection is just the Bott connection on $H(SM)$ in the theory of foliation (c.f. \cite{Zhang}).
\end{remark}

Let $R^{\rm Ch}=\left(\nabla^{\rm Ch}\right)^2$  be the curvature of $\nabla^{\rm Ch}$.
Let $\Omega=\left(\Omega_j^i\right)$ be the curvature forms of $R^{\rm Ch}$. From the torsion freeness, the curvature form has no pure vertical differential form
\begin{align*}
\Omega_j^i=\frac{1}{2}R_{j~kl}^{~i}\omega^k\wedge\omega^l+P_{j~k\gamma}^{~i}\omega^k\wedge\omega^{n+\gamma}.
\end{align*}
The Landsberg curvature is defined as
$$\mathbf{L}:=L^{~i}_{j~\gamma}\omega^{n+\gamma}\otimes\omega^j\otimes\mathbf{e}_i=-P_{j~n\gamma}^{~i}\omega^{n+\gamma}\otimes\omega^j\otimes\mathbf{e}_i,$$
the mean Landsberg curvature is defined by
$$\mathbf{J}={\rm tr}\mathbf{L}=J_{\gamma}\omega^{n+\gamma}=-P_{i~n\gamma}^{~i}\omega^{n+\gamma}.$$
If a Finsler manifold satisfies $P=0$, $\mathbf{L}=0$ or $\mathbf{J}=0$, then it is called a Berwald, Landsberg or weak Landsberg manifold, respectively.

As we will study the geometry of the fibers of $SM$, in the following we mention some facts of the geometry of the unit sphere bundle which is built around the theory of the Chern connection.

Let $\pi:SM\to M$ be the unit sphere bundle. The tensor
\begin{align*}
g^{T(SM)}=\sum_{i=1}^{n}\omega^i\otimes\omega^i+\sum_{\alpha=1}^{n-1}\omega^{n+\alpha}\otimes\omega^{n+\alpha}=:g+\dot{g}
\end{align*}
gives raise a Sasaki-type Riemannian metric on $SM$. Using the adapted frame and the Chern connection, the connection matrix of the Levi-Civita connection $\nabla^{T(SM)}$ of $g^{T(SM)}$ is given by (cf. \cite{Mo} )
\begin{equation*}
\Theta=\left[ \begin{array}{cc}
\omega_j^i+\left(H_{ij\gamma}+\frac{1}{2}R_{n~ij}^{~\gamma}\right)\omega^{n+\gamma}&-\left(\frac{1}{2}R_{n~ij}^{~\alpha}+H_{ij\alpha}\right)\omega^{j}-P_{n~i\gamma}^{~\alpha}\omega^{n+\gamma}\\
\left(H_{ij\alpha}+\frac{1}{2}R_{n~ij}^{~\alpha}\right)\omega^{j}+P_{n~i\beta}^{~\alpha}\omega^{n+\beta}&\omega_{\beta}^\alpha+H_{\alpha\beta\gamma}\omega^{n+\gamma}
\end{array}\right].
\end{equation*}

As the restriction the Levi-Civita connection $\nabla^{T(SM)}$ on $\mathcal{F}$, $\nabla^{\mathcal{F}}:=p\nabla^{T(SM)}p$ is the Euclidean connection of the bundle $(\mathcal{F},\dot{g})$. It is clear that along each fiber of $S_xM$, $x\in M$, $\nabla^{\mathcal{F}}$ is just the Levi-Civita connection of the Riemannian manifold $(S_xM, \dot{g}_x)$.

For any $x\in M$, the fiber of $S_xM$ is the indicatrix of the Minkowski space $(T_xM, F_x)$, where $F_x:=F|_{T_xM}$ for simplicity. By the discussion in our previous paper \cite{Li}, $\dot{g}_x$ is just the centro-affine metric of $S_xM$, when it is considered as an affine hypersurface in $T_xM$. We define a connection on $\mathcal{F}$ by $\bar{\nabla}^{\rm Ch}:=\nabla^{\mathcal{F}}+J\circ H\circ J$. Therefore,  the connection forms of $\bar{\nabla}^{\rm Ch}$ are precisely given by
\begin{equation}\label{vertical ch connection}
  \bar{\nabla}^{\rm Ch}\mathbf{e}_{n+\beta}=\omega_{\beta}^{\alpha}\otimes\mathbf{e}_{n+\alpha}.
\end{equation}
In fact $\bar{\nabla}^{\rm Ch}$ gives the affine connection along the fibers of $SM$.

We introduces some symbols to denote different covariant differentials for conveniens . For example, let $T=T^i_j\mathbf{e}_i\otimes\omega^j$ be an arbitrary smooth local section of the bundle $H(SM)\otimes H^*(SM)$ over $SM$. Then
\begin{equation*}
  \nabla^{\rm Ch}T=\left(dT^i_j+T^k_j\omega^i_k-T^i_k\omega^k_j\right)\otimes\mathbf{e}_i\otimes\omega^j
\end{equation*}
If we expand the coefficients as one forms on $SM$ in terms of the adapted coframe, then we denote
\begin{equation*}
(\nabla^{\rm Ch}T)^i_j:=dT^i_j+T^k_j\omega^i_k-T^i_k\omega^k_j=:T^i_{j|\alpha}\omega^\alpha+T^i_{j|n}\omega^n+T^i_{j;\gamma}\omega^{n+\gamma}.
\end{equation*}
Therefore we obtain
\begin{equation*}
 T^i_{j|\alpha}=i_{\mathbf{e}_{\alpha}}(\nabla^{\rm Ch}T)^i_j,\quad
 T^i_{j|n}=i_{\mathbf{e}_{n}}(\nabla^{\rm Ch}T)^i_j, \quad
 T^i_{j;\gamma}=i_{\mathbf{e}_{n+\gamma}}(\nabla^{\rm Ch}T)^i_j
\end{equation*}
where $i_{v}$ is the notation for the interior multiplication on differential forms by any vector $v$.

According to the splitting (\ref{splitting}) and using the almost complex structure $J$, we obtain a section of $\mathcal{F}\otimes\mathcal{F}^*$ from $T$ as following
\begin{equation*}
  \bar{T}=T^{\alpha}_{\beta}\mathbf{e}_{n+\alpha}\otimes\omega^{n+\beta}.
\end{equation*}
By (\ref{vertical ch connection}), the covariant differential of $\bar{T}$ by using $\bar{\nabla}^{\rm Ch}$ is given by
\begin{equation*}
  \bar{\nabla}^{\rm Ch}\bar{T}=\left(dT^{\alpha}_{\beta}+T^{\mu}_{\beta}\omega^{\alpha}_{\mu}-T^{\alpha}_{\mu}\omega^{\mu}_{\beta}\right)\otimes\mathbf{e}_{n+\alpha}\otimes\omega^{n+\beta}
\end{equation*}
Similarly, we denote that
\begin{equation*}
  (\bar{\nabla}^{\rm Ch}\bar{T})^{\alpha}_{\beta}:=dT^{\alpha}_{\beta}+T^{\mu}_{\beta}\omega^{\alpha}_{\mu}-T^{\alpha}_{\mu}\omega^{\mu}_{\beta}=T^{\alpha}_{\beta\parallel\gamma}\omega^{\gamma}+T^{\alpha}_{\beta\parallel n}\omega^{n}+T^{\alpha}_{\beta,\gamma}\omega^{n+\gamma}.
\end{equation*}
Therefore, by Lemma \ref{sturcture eq}, one has the following relations
\begin{equation}\label{covar ch bar ch}
  \begin{split}
    T^{\alpha}_{\beta\parallel i}&=i_{\mathbf{e}_{i}}(dT^{\alpha}_{\beta}+T^{\mu}_{\beta}\omega^{\alpha}_{\mu}-T^{\alpha}_{\mu}\omega^{\mu}_{\beta})=i_{\mathbf{e}_{i}}(dT^{\alpha}_{\beta}+T^{k}_{\beta}\omega^{\alpha}_{k}-T^{\alpha}_{k}\omega^{k}_{\beta})= T^{\alpha}_{\beta| i}\\
    T^{\alpha}_{\beta,\gamma}&=i_{\mathbf{e}_{n+\gamma}}(dT^{\alpha}_{\beta}+T^{\mu}_{\beta}\omega^{\alpha}_{\mu}-T^{\alpha}_{\mu}\omega^{\mu}_{\beta})\\
    &=i_{\mathbf{e}_{n+\gamma}}(dT^{\alpha}_{\beta}+T^{k}_{\beta}\omega^{\alpha}_{k}-T^{\alpha}_{k}\omega^{k}_{\beta})-i_{\mathbf{e}_{n+\gamma}}T^n_{\beta}\omega_n^{\alpha}+i_{\mathbf{e}_{n+\gamma}}T^{\alpha}_n\omega^n_{\beta}\\
    &= T^{\alpha}_{\beta;\gamma}+T^n_{\beta}\delta^\alpha_\gamma-T_n^\alpha\delta_{\beta\gamma}.
  \end{split}
\end{equation}
It is obvious that if $T$ is a section $J\mathcal{F}\otimes (J\mathcal{F})^*$, then $T^{\alpha}_{\beta;\gamma}=T^{\alpha}_{\beta,\gamma}$. These relations will be used in the following study.

Another important linear connection in Finsler geometry is the Berwald connection. We will present some relations between the curvatures of the Chern connection and the Berwald curvatures.
The Berwald connection is defined by $\nabla^{\rm B}=\nabla^{\rm Ch}+J^*\mathbf{L}$, where $J^*$ is dual of the almost complex structure $J$. Let $\tilde{\bm{\omega}}=(\tilde{\omega}_j^i)$ denote the Berwald connection form, then
\begin{align*}
  \tilde{\omega}_j^i=\omega_j^i-L^i_{j\alpha}\omega^\alpha.
\end{align*}

From the torsion freeness the Chern connection and the following well known formula (c.f. \cite{BaoChernShen,ChernShen,Mo,SS})
\begin{align*}
P_{n~k\gamma}^{~i}=-H_{ki\gamma|n}=-L_{ik\gamma},
\end{align*}
the Berwald connection is torsion free,
\begin{equation*}
  d\omega^i=\omega^j\wedge\tilde{\omega}_j^i.
\end{equation*}
Let  $\tilde{\Omega}_j^i$ be the curvature forms of the Berwald connection.
By the torsion freeness of the Berwald connection,  we have
\begin{align*}
\tilde{\Omega}_j^i=\frac{1}{2}\tilde{R}_{j~kl}^{~i}\omega^k\wedge\omega^l+\tilde{P}_{j~k\gamma}^{~i}\omega^k\wedge\omega^{n+\gamma},
\end{align*}
where $\tilde{R}_{j~kl}^{~i}=-\tilde{R}_{j~lk}^{~i}.$ By using Lemma \ref{sturcture eq}, we have the formulae of the curvatures of the Berwald connection,
\begin{align}\label{curvature Berwald by Chern}
\begin{split}
  &\tilde{R}^{~\alpha}_{\beta~\gamma\mu}=R^{~\alpha}_{\beta~\gamma\mu}-(L^\nu_{\beta\gamma}L^\alpha_{\nu\mu}-L^\nu_{\beta\mu}L^\alpha_{\nu\gamma})+(L^\alpha_{\beta\gamma|\mu}-L^\alpha_{\beta\mu|\gamma}),\\
  &\tilde{R}^{~\alpha}_{\beta~\gamma n}= R^{~\alpha}_{\beta~\gamma n},\quad \tilde{R}^{~\alpha}_{n~kl}=-\tilde{R}^{~n}_{\alpha~kl}=R^{~\alpha}_{n~kl},\\
  &\tilde{P}^{~\alpha}_{\beta~\gamma\mu}=P^{~\alpha}_{\beta~\gamma\mu}+L^\alpha_{\beta\gamma;\mu},\quad \tilde{P}^{~n}_{\alpha~\gamma\mu}=2L_{\alpha\gamma\mu},\\
  &\tilde{P}^{~\alpha}_{\beta~n\mu}=0,\quad \tilde{P}^{~\alpha}_{n~k\gamma}=0,\quad \tilde{P}^{~n}_{\alpha~n\gamma}=0.
  \end{split}
\end{align}

Using the explicit formulae of the connections and the curvature tensors under natural coordinate systems (cf. \cite{BaoChernShen}. pp. 27-67), one finds that
\begin{equation*}
  \tilde{P}_{j~kl}^{~i}=F\frac{\partial^3G^i}{\partial y^j\partial y^k\partial y^l}=:FB^i_{jkl}
\end{equation*}
is the usual Berwald curvature $\mathbf{B}$. It is clear that the curvature $P=0$ of the Chern connection if and only if $\mathbf{B}=0$. The mean Berwald curvature or the $\mathbf{E}$-curvature is defined by
\begin{equation*}
  \mathbf{E}=E_{\gamma\mu}\omega^\gamma\otimes\omega^{n+\mu}=\tr \tilde{P}.
\end{equation*}
Under the local adapted frame, the coefficients of the $\mathbf{E}$-curvature is represented as
\begin{equation}\label{coefficients E}
  E_{\gamma\mu}=\tilde{P}^{~i}_{i~\gamma\mu}=\tilde{P}^{~\alpha}_{\alpha~\gamma\mu}=P^{~\alpha}_{\alpha~\gamma\mu}+L^\alpha_{\alpha\gamma;\mu}=P^{~\alpha}_{\alpha~\gamma\mu}+J_{\gamma;\mu}=P^{~\alpha}_{\alpha~\gamma\mu}+J_{\gamma,\mu}.
\end{equation}

\section{A vertical elliptic equation of the S-curvature}


On a local coordinate chart $(U;x^i)$, let $dV_M=\sigma(x)dx^1\wedge\cdots\wedge dx^n$
be any volume form on $M$.  The following important function
on $SM$ is well defined,
$$\tau=\ln\frac{\sqrt{\det{g_{ij}}}}{\sigma(x)}.$$
$\tau$ is called the distortion of $(M,F)$. $\tau$ is a very
important invariant of the Finsler manifold, which is first
introduced by Zhongmin Shen. One refers to \cite{ChernShen,SS}
for more discussion about the distortion $\tau$. The restriction of $\tau$ on each fiber of $SM$ is also critical for the investigation of the centro-affine differential geometry of the fiber. One consults \cite{SSV} for more details.

By definition, the Cartan-type form $\eta$ is the vertical differential of the distortion.
\begin{lem}\label{lem deta}
The Cartan-type one form is the vertical differential of the distortion $\tau$.
$$d^{V}\tau=\eta,$$
where $d^{V}:=\omega^{n+\alpha}\wedge\D{T^*(SM)}_{\mathbf{e}_{n+\alpha}}$, $\D{T^*(SM)}$ is the dual connection of the Levi-Civita connection $\D{T(SM)}$.
\end{lem}

The following property of the Cartan-type form is critical for us.
\begin{prop}[\cite{Li}]
The exterior differentiation of $\eta$ is given by
\begin{equation*}
d\eta=-\tr [R^{\rm Ch}].
\end{equation*}
Then $d\eta$ has the local formula
\begin{equation}\label{deta}
d\eta=d(H_{\gamma}\omega^{n+\gamma})=-\frac{1}{2}R_{i~kl}^{~i}\omega^k\wedge\omega^l-P_{i~k\gamma}^{~i}\omega^k\wedge\omega^{n+\gamma}.
\end{equation}
\end{prop}

From Lemma \ref{lem deta}, one has the following corollary.
\begin{mcor}
The differential of $\tau$ is given by
\begin{equation}
d\tau=\tau_{|i}\omega^i+\eta=\tau_{|\alpha}\omega^\alpha+\tilde{S}\omega^n+\eta,\label{dtau}
\end{equation}
where we denote $\tau_{|i}:=\mathbf{e}_i (\tau)$, and $\tilde{S}:=\tau_{|n}=\mathbf{e}_n (\tau)$.
\end{mcor}

It is clear that $\tau_|:=\tau_{|i}\omega^i$ defines a section of the dual bundle of $J\mathcal{F}$. Therefore $\bar{\tau}_|:=\tau_{|\alpha}\omega^{n+\alpha}$ is a well defined section of the dual bundle $\mathcal{F}^*$ of the vertical bundle $\mathcal{F}$.

In literature, $\mathbf{S}:=F\tilde{S}=\mathbf{G}(\tau)$ defined on $TM_0$ is called the $S$-curvature of the Finsler manifold $(M,F)$.
The $S$-curvature is also introduced by Zhongmin Shen. For more results related to the
$S$-curvature, one refers to \cite{ChernShen,SS,Shenbook1,Shenbook2,Shen} and the references in them.

Taking the exterior differentiation of (\ref{dtau}), we obtain
\begin{equation}\label{d2tau1}
\begin{split}
0=&d^2\tau=d(\tau_{|i}\omega^i)+d\eta=d\tau_{|i}\wedge\omega^i+\tau_{|i}d\omega^i+d\eta=(d\tau_{|i}-\tau_{|j}\wedge\omega_i^j)\wedge\omega^i+d\eta\\
=&(\nabla^{\rm Ch}\tau_{|})_i\wedge\omega^i+d\eta
=\tau_{|i|j}\omega^j\wedge\omega^i-\tau_{|i;\gamma}\omega^i\wedge\omega^{n+\gamma}+d\eta\\
=&\tau_{|\alpha|\beta}\omega^\beta\wedge\omega^\alpha+\tau_{|n|\beta}\omega^\beta\wedge\omega^n+\tau_{|\alpha|n}\omega^n\wedge\omega^\alpha
-\tau_{|\alpha;\gamma}\omega^\alpha\wedge\omega^{n+\gamma}\\
&-\tau_{|n;\gamma}\omega^n\wedge\omega^{n+\gamma}+d\eta,
\end{split}
\end{equation}
where $\tau_{|i|j}$ and $\tau_{|i;\gamma}$ denote the coefficients of the covariant differential of $\tau_|$ with respect to the Chern connection.

One easily verifies the following identities
\begin{equation}\label{coefficients}
\begin{split}
  &\tau_{|n|\alpha}:=i_{\mathbf{e}_\alpha}(d\tau_{|n}-\tau_{|i}\omega_n^i)=i_{\mathbf{e}_\alpha}(d\tilde{S}-\tau_{|\beta}\omega_n^\beta)=\mathbf{e}_\alpha(\tilde{S})=\tilde{S}_{|\alpha},\\
  &\tau_{|n;\alpha}:=i_{\mathbf{e}_{n+\alpha}}(d\tau_{|n}-\tau_{|i}\omega_n^i)=\mathbf{e}_{n+\alpha}(\tilde{S})+\tau_{|\beta}\delta^\beta_\alpha=\tilde{S}_{,\alpha}+\tau_{|\beta}\delta^\beta_\alpha,\\
  &\tau_{|\alpha;\beta}:=i_{\mathbf{e}_{n+\alpha}}(d\tau_{|\alpha}-\tau_{|i}\omega_\alpha^i)=i_{\mathbf{e}_{n+\alpha}}(d\tau_{|\alpha}-\tau_{|\gamma}\omega_\alpha^\gamma)-\tilde{S}\delta_{\alpha\beta}=\tau_{|\alpha,\beta}-\tilde{S}\delta_{\alpha\beta},
\end{split}
\end{equation}
where $\tau_{|\alpha,\beta}$ denote the vertical coefficients of $\bar{\nabla}^{\rm Ch}\bar{\tau}_{|}$.

Plugging (\ref{deta}) and (\ref{coefficients}) into (\ref{d2tau1}), we have
\begin{equation}\label{d2tau2}
  \begin{split}
    0=&-\frac{1}{2}(\tau_{|\alpha|\beta}-\tau_{|\beta|\alpha})\omega^\alpha\wedge\omega^\beta+\tilde{S}_{|\alpha}\omega^\alpha\wedge\omega^n+\tau_{|\alpha|n}\omega^n\wedge\omega^\alpha\\
    &-(\tau_{|\gamma,\mu}-\tilde{S}\delta_{\gamma\mu})\omega^\gamma\wedge\omega^{n+\mu}
    -(\tilde{S}_{,\mu}+\tau_{|\mu})\omega^n\wedge\omega^{n+\mu}
    -\frac{1}{2}R_{i~\alpha\beta}^{~i}\omega^\alpha\wedge\omega^\beta\\
    &-R_{i~\alpha n}^{~i}\omega^\alpha\wedge\omega^n-P_{i~n\mu}^{~i}\omega^n\wedge\omega^{n+\mu}-P_{i~\gamma\mu}^{~i}\omega^\gamma\wedge\omega^{n+\mu}.
  \end{split}
\end{equation}
The following identities are derived from (\ref{d2tau2})
\begin{align}
  &\tau_{|\gamma,\mu}-\tilde{S}\delta_{\gamma\mu}+P_{i~\gamma\mu}^{~i}=0,\label{basic equ 3}\\
  &\tilde{S}_{,\mu}+\tau_{|\mu}+P_{i~n\mu}^{~i}=0.\label{basic equ 4}
\end{align}
Substituting the formula of $\tau_{|\mu}$ from (\ref{basic equ 4}) into (\ref{basic equ 3}), we obtain by using (\ref{coefficients E})
\begin{equation}\label{hessian S}
  \tilde{S}_{,\gamma,\mu}+\tilde{S}\delta_{\gamma\mu}=P_{i~\gamma\mu}^{~i}-P_{i~n\gamma,\mu}^{~i}=P_{i~\gamma\mu}^{~i}+J_{\gamma,\mu}^{~i}=E_{\gamma\mu}.
\end{equation}

Using natural frames, the formula (\ref{hessian S}) is known as the relation between the S-curvature and the $\mathbf{E}$-curvature as below
\begin{equation}\label{S E natural frame}
  F\cdot \mathbf{S}_{y^iy^j}=E_{ij}.
\end{equation}

Now we present a proof of Theorem \ref{theo 1}.
\begin{proof}[Proof of Theorem \ref{theo 1}]
  Using the Levi-Civita connection along each fiber of $SM$, the equation (\ref{hessian S}) is reformulated as
  \begin{equation}\label{hessian S'}
    {\rm Hess}_{\dot{g}}\tilde{S}_{\gamma\mu}+\sum_{\nu}H_{\gamma\mu\nu}\tilde{S}_{,\nu}+\tilde{S}\delta_{\gamma\mu}=E_{\gamma\mu},
  \end{equation}
  where ${\rm Hess}_{\dot{g}}$ denotes the Hessian operator on each fiber of $SM$ with respect to the metric $\dot{g}$.
  By taking trace, we obtain from ({\ref{hessian S'}}) the following linear elliptic equation
  \begin{equation}\label{laplace 1}
    \Delta_{\dot{g}}\tilde{S}+\dot{g}(\eta,d^V\tilde{S})+(n-1)\tilde{S}=\mathsf{e}.
  \end{equation}
  In \cite{Li}, we elaborated that $\eta=-(n-1)\hat{T}$, where $\hat{T}$ is the Tchebychev form in centro-affine differential geometry. Therefore (\ref{laplace 1}) along each fiber is
   \begin{equation}\label{laplace 2}
    \Delta_{\dot{g}}\tilde{S}-(n-1)\dot{g}(\hat{T},d^V\tilde{S})+(n-1)\tilde{S}=\mathsf{e}.
  \end{equation}
  Assume that $\mathsf{e}$ is a function on $M$, then $f=\tilde{S}-\frac{1}{n-1}\mathsf{e}$ solves the following equation along each fiber of $SM$
  \begin{equation}\label{Schrodinger equ}
    \Delta_{\dot{g}}f-(n-1)\dot{g}(\hat{T},d^Vf)+(n-1)f=0.
  \end{equation}
  Let $S_xM$ and $S^*_xM$ be the indicatrices of the Minkowski spaces $(T_xM,F_x)$ and $(T^*_xM,F^*_x)$, respectively, for any $x\in M$.
  It is known that the Legendre transformation $\mathfrak{L}_x$ gives a diffeomorphism between $S_xM$ and $S^*_xM$,  which preserves the affine metrics and changes the sign of the Tchebychev forms (cf. \cite{BaoChernShen}).
  By Satz 3.1 in \cite{Schneider1}, due to Blaschke and Schneider, the solutions of the equation (\ref{Schrodinger equ}) on $S^*_xM$ are of the form
  \begin{equation}
    \langle\xi_x,d^{T^*_xM}F^*_x\rangle|_{S^*_xM},
  \end{equation}
  where $\xi_x=\xi_i(x)\frac{\partial}{\partial p_i}\in T^*_xM$  are constant co-vectors.
  The differential of the norm $F^*_x$ is simply given by $$d^{T^*_xM}F^*_x=F^*_{p_i}(x,p)dp_i=\frac{1}{F^*_x(p)}g^{ij}(p)p_jdp^i.$$
  Thus the solutions of (\ref{Schrodinger equ}) on $S_xM$ are given by
  \begin{align*}
    f&=\mathfrak{L}^*\langle\xi_x,d^{T^*_xM}F^*_x\rangle=\mathfrak{L}^*\left[\frac{1}{F^*_x(p)}\xi_i(x)g^{ij}(p)p_j\right]\\
    &=\frac{1}{F^*_x(\mathfrak{L}_xy)}\xi_i(x)g^{ij}(\mathfrak{L}_xy)(\mathfrak{L}_xy)_j=\frac{1}{F_x(y)}\xi_i(x)y^i.
  \end{align*}
  Therefore the S-curvature is weakly isotropic
  \begin{equation}
   \mathbf{S}=F\tilde{S}=F\left(f+\frac{1}{n-1}\mathsf{e}\right)=\frac{1}{n-1}F\mathsf{e}+\xi_iy^i.
  \end{equation}
  By differentiation
  \begin{equation}
    \mathbf{S}_{y^iy^j}=\frac{1}{n-1}\mathsf{e}F^{-1}\mathbf{h}_{ij}.
  \end{equation}
  Recall (\ref{S E natural frame}), we complete the proof.
\end{proof}

\begin{remark}
  After the manuscript was finished and announced online, Professor Crampin kindly informed us his works in \cite{Cra2}. Inspired by our early work \cite{Li}, he first
  proved a global result of Minkowski norms, which is equivalent to (\ref{Schrodinger equ}). Using this result he proved that $\mathsf{e}=0$ implies $\mathbf{E}=0$. Then he proved a landsberg space for which the Berwald scalar curvature vanishes is a Berwald space.

  Like our early works in \cite{Li}, we adopt different methods from \cite{Cra2} and emphasize the geometry of fibers of the unit tangent bundle. We hope this viewpoint will be helpful for further studies.
\end{remark}


\begin{thebibliography}{99}












\bibitem{BaoChernShen} David Bao, Shiing-Shen Chern and Zhongmin Shen,
\textsl{An Introduction to Riemann-Finsler Geometry.} Graduate Texts
in Mathematics, Vol. 200, Springer-Verlag, New York, Inc., 2000.










\bibitem{ChernShen} Shiing-Shen Chern and Zhongmin Shen, \textsl{Riemann-Finsler Geometry.} Nankai Tracts in Mathematics, Vol. 6, World Scientific, 2005.

\bibitem{ChenHePan} Guangzu Chen, Qun He and Shengliang Pan, \textsl{On weak Berwald $(\alpha,\beta)$-metrics of scalar flag curvature}. Journal of Geometry and Physics, Vol. 86, 2014: 112-121.


\bibitem{Cra2} Mike Crampin, \textsl{A condition for a Landsberg space to be Berwaldian.}
Publ. Math. Debrecen, Vol. 93, no. 1-2, 2018: 143-155.



\bibitem{FL} Huitao Feng and Ming Li, \textsl{Adiabatic limit and connections in Finsler geometry.}
Communications in Analysis and Geometry, Vol. 21, No. 3, 2013: 607-624.


























\bibitem{Li} Ming Li, \textsl{Equivalence theorems of Minkowski spaces and applications in Finsler geometry.}(in Chinese) Acta Math. Sinica (Chin. Ser.) Vol. 62, No.2, 2019: 177-190. (see arXiv:1504.04475v2  for the English version.)
















\bibitem{Mo} Xiaohuan Mo, \textsl{An Introduction to Finsler Geometry.} Peking University series in Math., Vol. 1, World Scientific Publishing Co. Pte. Ltd., 2006.




\bibitem{NS} Katsumi Nomizu and Takeshi Sasaki, \textsl{Affine Differential Geometry.} Cambridge University Press, 1994.



\bibitem{Schneider1} Rolf Schneider, \textsl{Zur affinen Differentialgeometrie im Gro{\ss}en. I.} Math. Zeitschr., Vol. 101, 1967: 375-406.






\bibitem{SS} Yibing Shen and  Zhongmin Shen, \textsl{Introduction to Modern Finsler Geometry.} World Scientific, Singapore, 2016.


\bibitem{Shenbook1}Zhongmin Shen, \textsl{Differential Geometry of Spray and Finsler
Spaces.} Kluwer Acad. Publ., 2001.

\bibitem{Shenbook2}Zhongmin Shen, \textsl{Lectures on Finsler Geometry.} World Scientific, 2001.


\bibitem{Shen} Zhongmin Shen, \textsl{Landsberg curvature, S-curvature and Riemann curvature.} Riemann-Finsler Goemetry, MSRI Publications. Vol. 50, 2004: 303-355.






\bibitem{SSV} Udo Simon, Angela Schwenk-Schellschmidt and Helmut Viesel, \textsl{Introduction to the Affine Differential Geometry of Hypersurfaces.}
Lecture notes, Science University Tokyo, 1991. 


















\bibitem{Zhang} Weiping Zhang, \textsl{Lectures on Chern-Weil Theory and Witten Deformations}. Nankai Tracts in Mathematics, Vol. 4, World Scientific Publishing Co. Pte. Ltd., 2001.




\bibitem{ZouCheng14}Yangyang Zou and Xinyue Cheng, \textsl{The generalized unicorn problem on $(\alpha,\beta)$-metrics.} J. Math. Anal. Appl., Vol. 414, 2014: 574-589.

\end{thebibliography}
\end{document}